\documentclass[10pt]{amsart}

\usepackage{amssymb}
\usepackage{amsmath}
\usepackage{bbm}
\usepackage{psfrag}
\usepackage{color}
\usepackage{graphicx}
\usepackage{amsthm}

\DeclareMathOperator{\step}{\mathbbm{1}}

\newtheorem{theorem}{Theorem}[section] 
\newtheorem{lemma}[theorem]{Lemma}     

\newtheorem{definition}[theorem]{Definition}

\newtheorem{remark}[theorem]{Remark}

\begin{document}
\title[Functional calculus using systems theory]{Functional calculus for $C_0$-semigroups using infinite-dimensional systems theory
}
\author{Felix L. Schwenninger}
\address{Felix L. Schwenninger\\
Department of Applied Mathematics\\ P.O.\ Box 217, 7500 AE Enschede\\ The Netherlands}
\email{f.l.schwenninger@utwente.nl}
\date{5 February 2015}
\thanks{The first named author has been supported by the Netherlands Organisation for Scientific
Research (NWO), grant no. 613.001.004.}

\author{Hans Zwart}
\address{Hans Zwart\\ Department of Applied Mathematics\\ P.O.\ Box 217, 7500 AE Enschede\\ The Netherlands}
\email{h.j.zwart@utwente.nl}

\dedicatory{
\textit{
Dedicated to Charles Batty on the occasion of his sixtieth birthday.
}
}

\begin{abstract}
  In this short note we use ideas from systems theory to define a
  functional calculus for infinitesimal generators of strongly
  continuous semigroups on a Hilbert space. Among others, we show how
  this leads to new proofs of (known) results in functional calculus.
\end{abstract}


\maketitle

\section{Introduction}

  Let $A$ be a linear operator on the linear space $X$. In essence, a functional
  calculus provides for every (scalar) function $f$ in the algebra
  ${\mathcal A}$ a linear operator $f(A)$ from (a subspace of) $X$ to
  $X$ such that
 \begin{itemize}
 \item $f\mapsto f(A)$ is linear;
 \item $f(s)\equiv1$ is mapped on the identity $I$;
 \item If $f(s)=(s-r)^{-1}$, then $f(A)=(A-rI)^{-1}$;
 \item For $f=f_1\cdot f_2$ we have $f(A) = f_1(A) f_2(A)$.
 \end{itemize}
 As the domains of the operators $f(A)$ might differ, the above properties have to be seen formally, and, in general, need to be made rigorous. 
 It is well-known that self-adjoint (or unitary operators) on a Hilbert
 space have a functional calculus with ${\mathcal A}$ being the set of
 continuous functions from ${\mathbb R}$ (or the torus $\mathbb{T}$ respectively) to ${\mathbb C}$, (von Neumann
 \cite{Neum96}). This theory has been further extended to different
 operators and algebra's, see e.g.\ \cite{HiPh57}, \cite{DuSc71}, and
 \cite{AlDM96}. For an excellent overview, in particular on the $\mathcal{H}^{\infty}$-calculus, we refer to the book by
 Markus Haase, \cite{Haas06}. 

 For the algebra of bounded analytic functions on the left half-plane
 and $A$ the infinitesimal generator of a strongly continuous
 semigroup, we show how to build a
 functional calculus using infinite-dimensional systems theory.

\section{Functional calculus for ${\mathcal H}_{\infty}^-$}

We choose our class of functions to be ${\mathcal
  H}_{\infty}^-$, i.e., the algebra of bounded analytic functions on
the left half-plane. For $A$ we choose the generator of an
exponentially stable strongly continuous semigroup on
the Hilbert space $X$. This semigroup will be denoted by $\left(e^{At}\right)_{t\geq 0}$. We refer to \cite{EnNa00} for a detailed overview on  $C_{0}$-semigroups. In the following  all semigroups are assumed to be strongly continuous. To explain our choice/set-up we start with the following
observation.

  Let $h$ be an integrable function from ${\mathbb R}$ to ${\mathbb
    C}$ which is zero on $(0,\infty)$ and let $t\mapsto\step(t)$
  denote the indicator function of $[0,\infty)$, i.e., $\step(t)=1$ for $t\geq0$ and $\step(t)=0$ for $t<0$. Then for $t>0$
  \begin{eqnarray*}
     \left( h * e^{A\cdot}x_0\step(\cdot) \right) (t)&=&\int_{-\infty}^{\infty} h(\tau) e^{A(t-\tau)} x_0\step(t-\tau) d\tau \\
     &=& \left[\int_{-\infty}^t h(\tau) e^{-A\tau} d\tau \right] e^{At}x_0\\
    & = 
& 
\left[\int_{-\infty}^0 h(\tau) e^{-A\tau} d\tau \right] e^{At}x_0.
  \end{eqnarray*}
  Hence the convolution of $h$ with the semigroup gives an operator
  times the semigroup. We denote this operator by $g(A)$, with $g$ the
  Laplace transform of $h$.

  Now we want to extend the mapping $g \mapsto g(A)$. Therefore we
  need the Hardy space $H^2(X)=H^{2}(\mathbb{C}_{+};X)$, i.e., the set
  of $X$-valued functions, analytic on the right half-plane which are
  uniformly square integrable along every line parallel to the
  imaginary axis. By the (vector-valued) Paley-Wiener Theorem, this
  space is isomorphic to $L^2((0,\infty);X)$ under the Laplace
  transform, see \cite[Theorem 1.8.3]{ABHN}.
\begin{definition}
\label{D2.1}
  Let $X$ be a Hilbert space. For $g \in {\mathcal H}_{\infty}^-$ and
  $f \in L^2((0,\infty);X)$ we define the {\em Toeplitz operator}
  \begin{equation}
    \label{eq:2}
     M_g(f) = {\mathfrak L}^{-1} \left[ \Pi (g  \left({\mathfrak L}\left(f \right)\right) \right],
  \end{equation}
  where ${\mathfrak L}$ and ${\mathfrak L}^{-1}$ denotes the Laplace transform and its inverse, respectively, and $\Pi$ is the projection from $L^2(i{\mathbb R}, X)$ onto $H^2(X)$.
\end{definition}
\begin{remark}
  If we take $f(t)=e^{At}x_0$, $t \geq 0$, and ``$g={\mathfrak L}(h)$'', then this extends the previous convolution.
\end{remark}

The following norm estimate is easy to see.
\begin{lemma}
\label{L2.3} Under the conditions of Definition \ref{D2.1} we have that
$M_g$ is a bounded linear operator from $L^2((0,\infty);X)$ to itself
with norm satisfying
\begin{equation}
  \label{eq:3}
  \|M_g\| \leq \|g\|_{\infty}. 
\end{equation}
\end{lemma}

To show that Definition \ref{D2.1} leads to a functional calculus, we
need the following concept from infinite dimensional systems theory,
see e.g.\ \cite{Weis89}.
\begin{definition}
  Let $Y$ be a Hilbert space, and $C$ a linear operator bounded from $D(A)$,
  the domain of $A$, to $Y$. $C$ is an {\em admissible} output operator
  if the mapping $x_0 \mapsto C e^{A\cdot} x_0$
  can be extended to a bounded mapping from $X$ to
  $L^2([0,\infty);Y)$.
\end{definition}

Since in this paper only admissible output operators appear, we shall sometimes
omit ``output''. In \cite{Zwar12} the following was proved.
\begin{theorem}
\label{T2.4}
  Let $A$ be the generator of an exponentially stable semigroup on the Hilbert space $X$.
   For every $g \in {\mathcal H}_{\infty}^-$ there exists a linear mapping $g(A) : D(A) \mapsto X$ such that
  \[
      (\left(M_g(e^{A\cdot}x_0)\right)(t) = g(A) e^{At} x_0, \qquad x_0 \in D(A).
   \]
  Furthermore,
  \begin{itemize}
  \item $g(A)$ is an admissible operator;
    \item
       $g(A) e^{At}$ extends to a bounded operator for $t>0$;
\item
       $g(A)$ commutes with the semigroup;
\item $g(A)$ can be extended to a closed operator $g_{\Gamma}(A)$ such that $g\mapsto g_{\Gamma}(A)$ has the properties of an (unbounded) functional calculus; 
\item This (unbounded) calculus extends the Hille-Phillips calculus.
  \end{itemize}
\end{theorem}

Hence in general the functional calculus constructed in this way will
contain unbounded operators. However, they may not be 
``too unbounded'', as the product with any admissible operator is again
admissible. 
\begin{theorem}[Lemma 2.1 in \cite{Zwar12}]
\label{T2.6}
  Let $A$ be the generator of an exponentially stable semigroup on the
  Hilbert space $X$ and let $C$ be an admissible operator, then
  \[
      \left(M_g(C e^{A\cdot}x_0)\right)(t) = C g(A) e^{At} x_0, \qquad x_0 \in D(A^2).
   \]
  Moreover, $Cg(A)$ extends to an admissible output operator.
\end{theorem}

\section{Analytic semigroups}
  
\mbox{}From Theorem \ref{T2.4} we know that $g(A)e^{At}$ is a bounded
operator for $t>0$. In this section we show that for analytic
semigroups the norm of $g(A)e^{At}$ behaves like $|\log(t)|$ for $t$
close to zero. Let $A$ generate an exponentially stable, analytic
semigroup on the Hilbert space $X$. Then there exists a $M,\omega>0$ such
that, see \cite[Theorem 2.6.13]{Pazy83},
\begin{equation}
  \label{eq:1}
  \|(-A)^{\frac{1}{2}} e^{At}\| \leq M \frac{1}{\sqrt{t}} e^{-\omega
    t}, \qquad t >0.
\end{equation}
Using this inequality, we prove the following estimate.
\begin{theorem}
\label{T3.1}
  Let $A$ generate an exponentially stable, analytic semigroup on the
  Hilbert space $X$. There exists $m,\varepsilon_{0}>0$ such that for every $g \in {\mathcal H}_{\infty}^-$,
  $\varepsilon\in(0,\varepsilon_{0})$ 
  \begin{equation}
    \label{eq:4}
    \|g(A) e^{A\varepsilon}\| \leq m \|g\|_{\infty} |\log(\varepsilon)|.
  \end{equation}
If we assume that $(-A^*)^{\frac{1}{2}}$ or $(-A)^{\frac{1}{2}}$ is
admissible, then 
\begin{equation}
  \label{eq:5}
  \|g(A) e^{A\varepsilon}\| \leq m \|g\|_{\infty}
  \sqrt{|\log(\varepsilon)|} \quad \mbox{for} \quad \varepsilon\in(0,\varepsilon_{0}).
\end{equation}
If both $(-A^*)^{\frac{1}{2}}$ and $(-A)^{\frac{1}{2}}$ are
admissible, then $g(A)$ is bounded.
\end{theorem}
\begin{proof}
For $y \in D(A^*)$, $x \in D(A^2)$ we have 
\begin{align*}
  \frac{1}{2}\langle y, g(A) e^{A2\varepsilon} x \rangle &=
   \int_0^{\infty} \langle y, (-A) e^{A2t} g(A) e^{A2\varepsilon} x \rangle dt\\
   &=
  \int_0^{\infty} \langle (-A^*)^{\frac{1}{2}} e^{A^*\varepsilon} e^{A^* t} y , g(A) (-A)^{\frac{1}{2}}e^{A\varepsilon} e^{A t} x \rangle dt,
\end{align*}
where we used that $g(A)$ commutes with the semigroup.
Using Cauchy-Schwarz's inequality, we find
\begin{align}
\label{eq:6}
  \frac{1}{2}|\langle y, g(A) e^{A2\varepsilon} x \rangle| 
  &\leq
  \| (-A^*)^{\frac{1}{2}} e^{A^*\varepsilon} e^{A^*\cdot} y\|_{L^2} 
  \|g(A) (-A)^{\frac{1}{2}}e^{A\varepsilon} e^{A \cdot} x\|_{L^2}\\
\nonumber
  &=
   \| (-A^*)^{\frac{1}{2}} e^{A^*\varepsilon} e^{A^*\cdot} y\|_{L^2} \cdot
  \|M_g\left( (-A)^{\frac{1}{2}} e^{A\varepsilon} e^{A\cdot} x\right)
  \|_{L^2}\\
\nonumber
  &\leq \| (-A^*)^{\frac{1}{2}} e^{A^*\varepsilon} e^{A^*\cdot}
  y\|_{L^2} \cdot \|g\|_{\infty} \cdot
  \| (-A)^{\frac{1}{2}} e^{A\varepsilon} e^{A\cdot} x \|_{L^2},
\end{align}
where we used Lemma \ref{L2.3}. Hence it remains to estimate the two
$L^2$-norms. Since $X$ is a Hilbert space $\left(e^{A^*t}\right)_{t
  \geq 0}$ is an analytic semigroup as well. Hence both $L^2$-norms
behave similarly. We do the estimate for $e^{At}$. For $\omega
\varepsilon < 1/4$, 
\begin{align*}
  \| (-A)^{\frac{1}{2}} e^{A\varepsilon} e^{A\cdot} x \|_{L^2}^2 &=
  \int_0^{\infty} \| (-A)^{\frac{1}{2}} e^{A\varepsilon} e^{At} x \|^2
  dt \\
  &= \int_{\varepsilon}^{\infty} \| (-A)^{\frac{1}{2}} e^{At} x \|^2
  dt\\
  &\leq M^2 \int_{\varepsilon}^{\infty} \frac {e^{-2\omega t}}{t}
    \|x\|^2 dt \\
	&= M^2\|x\|^2 \int_{1}^{\infty} \frac {e^{-2\varepsilon\omega t}}{t}
    dt\\
  &\leq M^2 \|x\|^2 m_1 |\log(\varepsilon\omega)|,
\end{align*}
where we used (\ref{eq:1}) and $m_{1}$ is an absolute constant. 

Combining the estimates and using the fact that $\omega$ is fixed, we
find that there exists a constant $m_3>0$ such that for all $x\in
D(A^2)$ and $y \in D(A^*)$ there holds
\[
  |\langle y, g(A) e^{A2\varepsilon} x \rangle| \leq m_3 |\log(\varepsilon)| \|g\|_{\infty} \|x\|\|y\|.
\]
Since $D(A^2)$ and $D(A^*)$ are dense in  $X$, we have proved the
estimate (\ref{eq:4}).

We continue with the proof of inequality (\ref{eq:5}). If $
(-A^*)^{\frac{1}{2}}$ is admissible, then (\ref{eq:6}) implies that
\begin{align*}
   \frac{1}{2}|\langle y, g(A) e^{A2\varepsilon} x \rangle| 
  &\leq
  \| (-A^*)^{\frac{1}{2}} e^{A^*\varepsilon} e^{A^*\cdot} y\|_{L^2} 
  \|g(A) (-A)^{\frac{1}{2}}e^{A\varepsilon} e^{A \cdot} x\|_{L^2}\\
\nonumber
  &\leq
   m_2 \|y\| \cdot
  \|M_g\left( (-A)^{\frac{1}{2}} e^{A\varepsilon} e^{A\cdot} x\right)
  \|_{L^2}.
\end{align*}
The estimate follows as shown previously. Let us now assume that
$(-A)^{\frac{1}{2}}$ is admissible. Then by Theorem \ref{T2.6} there
holds
\begin{align*}
  \|g(A) (-A)^{\frac{1}{2}}e^{A\varepsilon} e^{A \cdot} x\|_{L^2} \leq
  &\ \|g(A) (-A)^{\frac{1}{2}} e^{A \cdot} x\|_{L^2}\\
  = &\ \|M_g \left((-A)^{\frac{1}{2}} e^{A \cdot} x\right)\|_{L^2} \\
  \leq&\ \|g\|_{\infty} \|(-A)^{\frac{1}{2}}e^{A \cdot} x \|_{L^2}\\
  \leq&\ \|g\|_{\infty} m \|x\|,
\end{align*}
where we have used Lemma \ref{L2.3} and the admissibility of
$(-A)^{\frac{1}{2}}$. Now the proof of (\ref{eq:5}) follows similarly
as in the first part. 

If $(-A)^{\frac{1}{2}}$ and $(-A^*)^{\frac{1}{2}}$ are both
admissible, then we see from the above that the epsilon disappears from
the estimate, and since the semigroup is strongly continuous, $g(A)$
extends to a bounded operator.
\end{proof}

In \cite{SchweZwa14a}, it is shown that for any $\delta\in(0,1)$ there exists an analytic, exponentially stable semigroup on a Hilbert space, and $g\in\mathcal{H}_{\infty}^{-}$ such that $(-A)^{\frac{1}{2}}$ is admissible and $\|g(A)
e^{A\varepsilon}\|\sim (\sqrt{|\log(\varepsilon)|})^{1-\delta}$. Similarly, the sharpness of (\ref{eq:4}) is shown.

In the next section we relate the above theorem to results in the
literature.

\section{Closing remarks}
\label{Sec:4}
A natural question is whether the calculus above coincides with other definitions of the $\mathcal{H}_{\infty}^{-}$-calculus. As the construction extends the Hille-Phillips calculus, the answer is ``yes'', see \cite{SchweZwa14b}.

In \cite{Vits05}, Vitse showed a similar estimate as in (\ref{eq:4}) for analytic semigroups on general Banach spaces by using the Hille-Phillips calculus. The setting there is slightly different since bounded analytic semigroups and functions $g\in\mathcal{H}_{\infty}^{-}$ with bounded Fourier spectrum are considered.
In \cite{SchweZwa14a}, the authors improve Vitse's result with a more direct technique. In the course of that work, the approach to Theorem \ref{T3.1} via the calculus construction used here was obtained. Moreover, the techniques here and in Vitse's work \cite{Vits05} require that the functions $f$ are bounded, analytic on a half-plane. 
In \cite{SchweZwa14a} it is shown that the corresponding result is even true for functions $f$ that are only bounded, analytic on sectors which are larger than the sectorality sector of the generator $A$. 

Furthermore, Haase and Rozendaal proved that (\ref{eq:4}) holds
for general (exponentially stable) semigroups on Hilbert spaces, see  \cite{HaRo13}. Their key tool is a \textit{transference principle}. More general, they show that on general Banach spaces one has to consider the analytic multiplier algebra $\mathcal{AM}_{2}(X)$, as the function space to obtain a corresponding result. Note that $\mathcal{AM}_{2}(X)$ is continuously embedded in $\mathcal{H}_{\infty}^{-}$ with equality if $X$ is a Hilbert space.

The difference in the transference principle and the approach followed
here is that in the transference principle, estimates are first proved
for ``nice'' functions and than extended to the whole space ${\mathcal H}_{\infty}^-$. Whereas we prove the result first for ``nice'' elements in $X$, and then extend the operators $g(A)$.

The fact that the calculus is bounded for analytic semigroups when
both $(-A)^{\frac{1}{2}}$ and $(-A^*)^{\frac{1}{2}}$ are admissible,
can  already be found in \cite{LeMe03}. However, as the admissibility of $(-A)^{\frac{1}{2}}$ is equivalent to $A$ satisfying \textit{square function estimates}, the result is much older and goes back to McIntosh, \cite{McIn86}.

The construction of the ${\mathcal H}_{\infty}^-$-calculus followed here can be adapted to general Banach spaces, see \cite{SchweZwa2012,SchweZwa14b}.

\end{document}